\documentclass[11 pt]{amsart}

\usepackage{lineno,hyperref}
\usepackage{amsmath,amssymb}
\usepackage{lineno}
\usepackage{centernot}
\theoremstyle{plain}
\newtheorem{thm}{Theorem}[section]
\newtheorem{lemma}[thm]{Lemma}
\newtheorem{cor}[thm]{Corollary}
\newtheorem{prop}[thm]{Proposition}

\theoremstyle{definition}
\newtheorem{defn}[thm]{Definition}
\newtheorem{example}[thm]{Example}

\newtheorem{rem}[thm]{Remark}

\numberwithin{equation}{section}

\newcommand{\interior}[1]{{\kern0pt#1}^{\mathrm{o}}}

\begin{document}


\title{NJ-symmetric Rings}

\author[Sanjiv Subba]{Sanjiv Subba $^\dagger$}

\address{$^\dagger$School of Applied Sciences\\ UPES, Dehradun\\  Bidholi 248007\\ India}
\email{sanjivsubba59@gmail.com}
\author[Tikaram Subedi]{Tikaram Subedi  {$^{\dagger *}$}}
\address{$^{\dagger * }$Department of Mathematics\\  National Institute Of Technology  Meghalaya\\ Shillong 793003\\ India}
\email{tikaram.subedi@nitm.ac.in}

.


\begin{abstract}
 We call a ring $R$ NJ-symmetric if $abc\in N(R)$ implies $bac\in J(R)$ for any $a,b,c\in R$. Some classes of rings that are NJ-symmetric include left (right) quasi-duo rings, weak symmetric rings, and abelian J-clean rings. We observe that if $R/J(R)$ is NJ-symmetric, then $R$ is NJ-symmetric, and therefore, we study some conditions for NJ-symmetric ring $R$ for which $R/J(R)$ is symmetric. It is observed that for any ring $R$, $M_n(R)$ is never an NJ-symmetric ring for all positive integer $n>1$. Therefore,  matrix extensions over an NJ-symmetric ring is studied in this paper. Among other results, it is proved that  there exists an NJ-symmetric ring whose  polynomial extension is not NJ-symmetric.

\end{abstract}

\maketitle

\section{INTRODUCTION}
\quad 	All rings considered in this paper are associative with identity unless otherwise mentioned. $R$ represents a ring, and all modules are unital. The symbols $Z(R),~E(R),~ J(R),~N(R)$, $U(R)$, $T_n(R)$, $M_n(R)$, $N^{*}(R)$, and $N_{*}(R)$ respectively denote the set of all central elements of $R$, the set of all idempotent elements of $R$, the Jacobson radical of $R$, the set of all nilpotent elements of $R$, the set of all units of $R$, the ring of all $n\times n$  upper triangular matrices  over $R$, the ring of all $n\times n$ matrices over $R$, the upper nil radical of $R$, and the lower nil radical of $R$. For any $a\in R$, the notation $l(a)$ $\left(r(a)\right)$ stands for the left (right) annihilator of $a$. 

\quad A ring is said to be reduced if it contains no nonzero nilpotent elements. In order to unify sheaf representation of commutative rings and reduced rings, Lambek in \cite{lambek} introduced the notion of symmetric rings (according to him a ring $R$ is \textit{symmetric} if $abc=0$ implies $bac=0$ for any $a,~b,~c\in R$).
A ring $R$ is said to be \textit{semicommutative} (also \textit{IFP} (see \cite{Basic} )) $R$ if $ab=0$ implies $aRb=0$ for $a,b\in R$. Ever since the introduction of symmetric rings and semicommutative rings, they are extensively studied (see \cite{olwar}, \cite{Basic}, \cite{lambek}, \cite{owsr}, \cite{ gwsr}) and continue to be an active area of research.
Furthermore, it was also establised that semicommutative rings are proper generalization of symmetric rings (see \cite{Basic}).  Motivated by this observation, we introduce the notion of NJ-symmetric rings by leveraging the idea of symmetric ring property to study a generalize semicommutative rings. It turns out that many well-known classes of rings, such as left (right) quasi-duo rings, and weak symmetric rings, fall under the class of NJ-symmetric rings.

 \quad Recall that $R$ is said to be:
\begin{enumerate}
	
	\item   \textit{weak symmetric } (\cite{owsr}) if $abc\in N(R)$ implies $acb\in N(R)$ for any $a,~ b,~ c\in R$.
	
	\item  \textit{generalized weakly symmetric (GWS)} (\cite{ gwsr})  if for any $a,~b,~c\in R$, $abc=0$ implies $bac\in N(R)$.
	
	\item \textit{left (right) quasi-duo} (\cite{1s2id}) if every maximal left (right) ideal of $R$ is an ideal of $R$.

	\item  \textit{semiperiodic} (\cite{sp}) if for each $a\in R\setminus(J(R)\cup Z(R))$, $a^q-a^p\in N(R)$ for some integers $p$ and $q$ of opposite parity.
	
	\item 	\textit{left} (\textit{right}) $SF$ (\cite{rnsf})   if all simple left (right) $R$-modules are flat. 
	
	\item 	\textit{2-primal} (\textit{reduced} ) if $N(R)=N_{*}(R)$ $\left(N(R)=0\right)$.
\end{enumerate}

\section{Main Results}

\begin{defn}
	We call $R$ an $NJ-symmetric$ if for any $a,b,c\in R$ and $abc\in N(R)$, then $bac\in J(R)$.  
\end{defn}
Following the definition of an NJ-symmetric ring, we obtain the following proposition.
\begin{prop}\label{equi}
	The following are equivalent:
	\begin{enumerate}
		\item R is NJ-symmetric.
		\item $abc\in N(R)$ implies $acb\in J(R)$ for any $a,b,c\in R$.
		\item $abc\in N(R)$ implies $cba\in J(R)$ for any $a,b,c\in R$.
	\end{enumerate}
\end{prop}

In a symmetric ring $R$, it is well-known that if $a_1a_2...a_n=0$, where $n$ is any positive integer, implies $a_{\sigma(r_1)}a_{\sigma(r_2)}...a_{\sigma(r_n)}=0$ for any permutation $\sigma$ of the set $\{1,2,..,n \}$ and $a_i\in R$.  So,  any symmetric ring is NJ-symmetric. However, the converse is not true (see the following example).
\begin{example}\label{eg}
	In view of Corollary \ref{upt}, $T_n(\mathbb{Z})$ is NJ-symmetric for any integer $n>1$, where $\mathbb{Z}$ is the ring of integers. Note that $e_{11}e_{21}=0$, $e_{21}e_{11}\neq 0$, where $e_{ij}$ denotes the matrix unit  in $T_n(\mathbb{Z})$ whose $(i,j)^{th}$ entry is $1$ and zero elsewhere. So, $T_n(\mathbb{Z})$ is not symmetric. 
\end{example}

$R$ is said to be \textit{J-clean} if for each $a\in R$, $a=e+j$ for some $e\in E(R)$ and $j\in J(R)$. If $E(R)\subseteq Z(R)$, then $R$ is said to be an abelian ring. An abelian ring need not be J-clean; for example, any field $F$ with $|F|>2$. 

For any $a\in R$,  
the \textit{commutant} of $a$ is defined by $comm(a)=\{y\in R~|~ ya=ay\}$ and $comm^2(a)=\{x\in R~|$ $yx=xy$ for all $y\in comm(a)  \}$ is called the  \textit{double commutant} of $a$.  $R$ is called $J-quasipolar$, if for any $a\in R$, $a+f\in J(R)$ for some $f^2=f\in comm^2(a)$.

\begin{thm}\label{qnjs}
	Each of the following classes of rings is NJ-symmetric:
	\begin{enumerate}
		\item \label{q} Left (right) quasi-duo rings.
		\item \label{j} Abelian J-clean rings.
		\item \label{jq} Abelian J-quasipolar rings.
		\item \label{g} GWS rings in which the index of nilpotent elements is at most 2. 
	\end{enumerate}

\end{thm}
\begin{proof}
	\begin{enumerate}
		\item Let $R$ be a left quasi-duo ring, and  $M$ be any maximal left ideal of $R$. Suppose  $a,b,c\in R$ be such that $abc\in N(R)$. Suppose that  $a\notin M$. Then, $M+Ra=R$. This implies that $m_1+r_1a=1$ for some $m_1\in M,~r_1\in R$. So, $m_1bc+r_1abc=bc$, which leads to $bc\in M$ as $N(R)\subseteq J(R)$ in left quasi-duo rings (\cite[Lemma 2.3]{oqdr}). Suppose $b\notin M$. Then, $M+Rb=R$. This implies that $m_2c+r_2bc=c$ for some $m_2\in M,~r_2\in R$. So, $c\in M$. Thus, $a\in M$ or $b\in M$ or $c\in M$. Therefore, $bac\in J(R)$. Hence, $R$ is NJ-symmetric. Similarly, it can be shown  that $R$ is an NJ-symmetric ring whenever $R$ is a right quasi-duo ring.
		
		\item \label{jc} Let $a,b,c\in R$ with $abc\in N(R)$. By J-cleanness, we have $a=e_1+j_1,~b=e_2+j_2,~c=e_3+j_3$, where $e_i\in E(R),~j_k\in J(R)$. Then, $abc=e_1e_2e_3+j_4$ for some $j_4\in J(R)$. This implies that $(abc)^n=(e_1e_2e_3+j_4)^n=0$ for some  positive integer $n$. This leads to $e_1e_2e_3\in E(R)\cap J(R)$ as $R$ is abalian. So, $e_1e_2e_3=0$. Therefore, $bac\in J(R)$.
		\item The proof is similar to (\ref{jc}).
		
		\item Let $abc\in N(R)$. Then, $barcb(abc)^2barcb=0$ for any $r\in R$. Since $R$ is GWS, $abcbarcbabcbarcb\in N(R)$. So, $abcbarcb\in N(R)$, that is, $bcbarcba\in N(R)$. So, $rc(bcbarcba)^2rc=0$. This implies that $bcbarcbarcbcbarcbarc\in N(R)$, that is, $(cbar)^2cb\in N(R)$. Since 
		
		$ar((cbar)^2cb)^2ar=0$, $cbar\in N(R)$. Hence, $cba\in J(R)$. By Proposition \ref{equi}, $R$ is NJ-symmetric. 
		
	\end{enumerate}
\end{proof}

\begin{example}\label{mat}
	Let $R$ be an arbitrary ring, and $n\geq 2$ be any integer. Take $a=e_{11}+e_{21},~b=e_{22},~c=e_{12}+e_{22}\in M_n(R)$, where $e_{ij}$ stands for the matrix unit in $M_n(R)$ whose $(i,j)^{th}$ entry is $1$ and zero elsewhere. Clearly, $abc=0\in N(M_n(R))$. Observe that $bac=e_{22}\notin J(M_n(R))$. So, $M_n(R)$ is not NJ-symmetric.
\end{example}

The condition ``GWS" for the ring $R$ in Theorem \ref{qnjs} (\ref{g}) cannot be dropped. By Example \ref{mat}, $M_2(\mathbb{Z}_2)$ is not NJ-symmetric. Observe that $M_2(\mathbb{Z}_2)$ is not GWS. Moreover, for any $x\in N(M_2(\mathbb{Z}_2))$, $x^2=0$.

There exists an NJ-symmetric ring that is not left (or right) quasi-duo  (see the following example).
\begin{example}
	By \cite[Example 2(ii)]{1s2id}, $\mathbb{H}[x]$ is not right quasi-duo, where $\mathbb{H}$ is the Hamilton quaternion over the field of real numbers.
	Since $\mathbb{H}[x]$ is reduced, it is NJ-symmetric. 
\end{example}

Nevertheless, in Theorem \ref{melt}, we will observe that an NJ-symmetric MELT ring is left (right) quasi-duo. A ring $R$ is said to be $MELT$ if every maximal essential left ideal is an ideal.

\begin{lemma}\label{meltlem}
	Let M be a maximal left (right) ideal of R, which is not essential. If R is NJ-symmetric, then M is an ideal.
\end{lemma}

\begin{proof}
	Let $R$ be an NJ-symmetric ring and $M$ a maximal left ideal of $R$, which is not essential. So, $M=l(e)$ for some $e~(\neq 0)\in E(R)$. Assume, if possible, that $M$ is not an ideal. Then, $Mr_0\not\subseteq M$ for some $r_0\in R$. So, $m_0r_0+m_1=1$ for some $m_0, m_1\in M$. This implies that $m_0r_0e+m_1e=e$. Since $M=l(e)$, $m_0re=e$. Observe that $r_0m_0e=0$. Since $R$ is NJ-symmetric, $m_0r_0e\in J(R)$. Hence, $e\in J(R)$, that is, $e=0$, a contradiction. Therefore, $M$ is an ideal of $R$.
	Similarly, one can prove the result for the right counterpart of $M$.
\end{proof}

\begin{thm}\label{melt}
	If $R$ is an NJ-symmetric MELT ring, then $R$ is left (right) quasi-duo.
\end{thm}

\begin{proof}
	Follows from Lemma \ref{meltlem}.
\end{proof}

\begin{prop}
Let  $R$ be an NJ-Symmetric ring. If $R$ is exchange, then:\\
\begin{enumerate}
    \item $R$ is clean.
    \item $R$ is quasi-duo.
\end{enumerate}
\end{prop}
\begin{proof}
\begin{enumerate}
    \item Let $\bar{e}\in E(R/J(R))$.  Since idempotents lifts modulo $J(R)$ (see \cite[Corollary 2.4]{lift}), we can assume that $e\in E(R)$. As $R$ is NJ-Abelian, $er(1-e)$ and $(1-e)re\in  J(R)$ for all $r\in R$. So, $re-er\in J(R)$, that is, $R/J(R)$ is Abelian. So,  $R/J(R)$ clean. Therefore, $R$ is clean (by \cite[Proposition 7]{eui}).
    \item By (1), $R/J(R)$ is Abelian. Hence, $R$ is quasi duo ( by \cite[Theorem 4.6]{quasi}
\end{enumerate}

  \end{proof}

We skip the proof of the following trivial lemma.
\begin{lemma}\label{wseqv}
	The following are equivalent:
	\begin{enumerate}
		\item $R$ is weak symmetric.
		\item $abc\in N(R)$ implies $bac\in N(R)$ for any $a,~b,~c\in R$.
	\end{enumerate}
\end{lemma}

In the following theorem, we use Lemma \ref{wseqv} freely.

\begin{thm}\label{bac}
	If $R$ is a weak symmetric ring, then $R$ in NJ-symmetric.
\end{thm}
\begin{proof}
	Let $a,b,c\in R$ be such that $(abc)^n=0$, where $n$ is a positive integer. Then, $\underbrace{ (abc)(abc)...(abc)}_\text{$n$-times}r^n=0$ for any $r\in R$. By Lemma \ref{wseqv}, $abcr\underbrace{(abc)(abc)....(abc)}_\text{$n-1$-times}r^{n-1}\in N(R)$, which leads to\\ $bacr\underbrace{(abc)(abc)...(abc)}_\text{$n-1$-times}r^{n-1}\in N(R)$. This further implies that \\ $abcrbacr\underbrace{(abc)(abc)...(abc)}_\text{$n-2$-times}r^{n-2}\in N(R)$. So, we obtain \\ $bacrbacr\underbrace{(abc)(abc)...(abc)}_\text{$n-2$-times}r^{n-2}\in N(R)$. Proceeding similarly, we obtain that $(bacr)^n\in N(R)$. So, $bac\in J(R)$. Therefore, $R$ is NJ-symmetric.
\end{proof}
\begin{rem}
    Semicommutative rings are NJ-symmetric (by \cite[Remark 1]{wsa}).
\end{rem}

There exists a ring that is NJ-symmetric but not weak symmetric, as illustrated in the following example.

\begin{example}\label{njbns}
  Let $K$ be a field and $D_n=K\langle x_n: x_n^{n+2}=0\rangle$, \\$R_n=\begin{pmatrix}
      D_n & x_nD_n\\
      x_nD_n & D_n
  \end{pmatrix}$. By \cite[Example 2]{wsa}  and \cite[Remark 1]{wsa}, $\prod\limits_{n=0}^{\infty}R_n$ is a quasi-duo ring which is not weak symmetric. Following Theorem \ref{qnjs} (\ref{q}), $\prod\limits_{n=0}^{\infty}R_n$ is NJ-symmetric.
\end{example}

\begin{thm}\label{rqj}
	If R/J(R) is NJ-symmetric, then $R$ is NJ-symmetric.
\end{thm}
\begin{proof}
	Let $a,~b,~c\in R$ be such that $abc\in N(R)$. Since $R/J(R)$ is NJ-symmetric, $\bar{b}\bar{a}\bar{c}\in J(R/J(R))$. So, $bac\in J(R)$.              
	
\end{proof}

The converse of Theorem \ref{rqj} is not true (see Example \ref{frac}). However, Theorem \ref{splsf} provides some conditions under which the converse holds.

\begin{example}\label{frac}
	Let $R$ be the localization of the ring $\mathbb{Z}$ at  $3\mathbb{Z}$ and $S$ the quaternions over $R$, that is, a free $R$-module with basis $1,~i,~j,~k$ and satisfy $i^2=j^2=k^2=-1,~ij=k=-ji$. Clearly, $S$ is a non-commutative domain. So, $S$ is an NJ-symmetric ring.  By \cite[Example 3]{yh},  $J(S)=3S$, and $S/3S\cong M_2(\mathbb{Z}_3)$. Take $x=\begin{pmatrix}
		\bar{0} & \bar{1}\\
		\bar{0} & \bar{0}
	\end{pmatrix}$ and $y=\begin{pmatrix}
		\bar{0} & \bar{0}\\
		\bar{1} & \bar{1}
	\end{pmatrix}$. Observe that $yx^2\in N({M_2(\mathbb{Z}_3)})$, and $xyx\notin J({M_2(\mathbb{Z}_3)})$ as $J(M_2(\mathbb{Z}_3))=0$. Hence, $S/J(S)$ is not NJ-symmetric. Moreover, this also implies that a homomorphic image of an NJ-symmetric need not be NJ-symmetric.
\end{example}

\quad It is obvious that reduced rings are NJ-symmetric. We can infer from Example \ref{frac} that $R/J(R)$ need not be reduced  when $R$ is NJ-symmetric. In this context, we have the following theorem.

\begin{thm}\label{splsf}
	Let $R$ is an NJ-symmetric ring. Then, $R/J(R)$ is reduced if:
	\begin{enumerate}
		\item $R$ is semiperiodic.
		\item $R$ is left SF.
	\end{enumerate}
\end{thm}

\begin{proof}
	
	\begin{enumerate}
		\item \label{dic} Let $R$ be an NJ-symmetric semiperiodic ring. Take $w\in R$ satisfying $w^2\in J(R)$. We shall show that $w\in J(R)$. Since $R$ is NJ-symmetric ring, $N(R)\subseteq J(R)$. Hence by \cite[Lemma 2.6]{sp}, either $J(R)=N(R)$ or $J(R)\subseteq Z(R)$. If $J(R)=N(R)$, then $w\in J(R)$. Suppose $N(R)\neq J(R)$. Then, $w^2\in Z(R)$. If $w\in Z(R)$, then $\bar{w}(R/J(R))\bar{w}=0$. As $R/J(R)$ is semiprime, $\bar{w}=0$, that is, $w\in J(R)$. Suppose $w\notin J(R)\cup Z(R)$.  By \cite[Lemma 2.3(iii)]{sp}, there exist $e\in E(R)$ and a positive integer $p$ satisfying $w^p=w^pe$ and $e=wy$ for some $y\in R$. Hence $e=ewy=ew(1-e)y+ewey=ew(1-e)y+ew^2y^2$. Note that $we(1-e)y=0$. As $R$ is NJ-symmetric, $ew(1-e)y\in J(R)$. Hence $e\in J(R)$, that is, $e=0$. This yields that $w^p=0$, a contradiction to $w\notin N(R)\subseteq J(R)$. Thus, $R/J(R)$ is reduced.
		
		\item 	Let $R$ be an NJ-symmetric left SF-ring. By \cite[Proposition 3.2]{rnsf}, $R/J(R)$ is left SF. Suppose there exists $w\in R$ such that $w^2\in J(R)$ and $w\notin J(R)$. Assume, if possible, $J(R)+Rr(w)=R$, then  $1=x+\sum\limits^{finite}r_is_i,~ x\in J(R),~ r_i\in R,~ s_i\in r(w)$. Then, $w=xw+\sum\limits^{finite}r_is_iw$. Observe that $s_ir_iw\in N(R)$. As $R$ is NJ-symmetric, $r_is_iw\in J(R)$. This implies that $w\in J(R)$, a contradiction. Hence, $J(R)+Rr(w)\neq R$. There exist some maximal left ideal $H$ satisfying $J(R)+Rr(w)\subseteq H$. Note that $w^2\in H$. By \cite[Lemma 3.14]{rnsf}, $w^2=w^2x$ for some $x\in H$, that is, $w-wx\in r(w)\subseteq H$. So, $w\in H$. Hence, there exists $y\in H$ satisfying $w=wy$, that is, $1-y\in r(w)\subseteq H$. This implies that $1\in H$, a contradiction. Therefore, $R/J(R)$ is reduced. 
	\end{enumerate} 
	
\end{proof}

\begin{cor}\label{c1}
	If $R$ is an NJ-symmetric semiperiodic ring, then $R/J(R)$ is commutative.
\end{cor}

\begin{proof}
	Since $R/J(R)$ is semiperiodic, by Theorem \ref{splsf} (\ref{dic}) and \cite[Theorem 4.4]{sp}, $R/J(R)$ is commutative.
\end{proof} 

$R$ is said to be regular ( strongly regular ) if for each $a\in R$, $a\in aRa ~(a\in a^2R)$.	
\begin{cor}\label{c2}
	If $R$ is an NJ-symmetric left SF semiperiodic ring,  then $R$ is a commutative regular ring.
\end{cor}
\begin{proof}
	By Theorem \ref{splsf}, $R/J(R)$ is reduced. Hence $R/J(R)$ is strongly regular by \cite[Remark 3.13]{rnsf}. This implies that $R$ is left quasi-duo, and hence by \cite[Theorem 4.10]{rnsf}, $R$ is strongly regular. Hence, $J(R)=0$. By Corollary \ref{c1}, $R$ is commutative.
\end{proof}

\begin{prop}\label{eRe}
	R is NJ-symmetric if and only if $eRe$ is NJ-symmetric for all $e\in E(R)$.
\end{prop}
\begin{proof}
	Suppose $R$ is NJ-symmetric. Let $eae,~ebe,~ece\in eRe$ with \\ $(eae)(ebe)ece\in N({eRe})$. Since $R$ is NJ-symmetric, $(ebe)(eae)(ece)\in J(R)$. Since $eJ(R)e=J(eRe)$, $(ebe)(eae)(ece)\in J({eRe})$. Hence, $eRe$ is NJ-symmetric. Whereas the converse is trivial.
\end{proof}

However, if every corner ring $eRe$ of $R$ is NJ-symmetric for all nonidentity idempotents $e$ of $R$, $R$ need not be NJ-symmetric as illustrated in the following example. 
\begin{example}

    Take $R=M_2(\mathbb{Z}_2)$. Then, $E(R)=\{0,I,E_{11}, E_{22}, E_{11}+E_{12}, \\ E_{11}+E_{21}, E_{22}+E_{12}, E_{22}+E_{21} \}$, where $I$ is the identity matrix of $R$. Observe that $E_{ii}RE_{ii}\cong \mathbb{Z}_2$, $(E_{ii}+E_{kl})R(E_{ii}+E_{kl})\cong \left\{ (E_{ii}+E_{kl})a:a\in \mathbb{Z}_2\right\}$, where $i,k,l\in \{1,2\}$ and $k\neq l.$ Clearly, if $e \in E(R)\setminus \{0,I\} $, then $eRe$ is domain, and hence NJ-symmetric. Clearly, $R$ is not NJ-symmetric.

\end{example}

In Example \ref{mat}, we obtain that for any ring $R$, $M_n(R)$ is not NJ-symmetric for all integers $n\geq 2$. So, we study the matrix extensions over  an NJ-symmetric ring in the following.

A \textit{Morita context} (\cite{morita}) is a $4$-tuple $\begin{pmatrix}
	R_1 & M\\
	P & R_2
\end{pmatrix}$, where $R_1$, $R_2$ are rings, $M$ is $(R_1,R_2)$-bimodule and $P$ is $(R_2,R_1)$-bimodule, and there exist context products $M\times P\rightarrow R_1$ and $P\times M\rightarrow R_2$ written multiplicatively as $(m,p)\mapsto mp$ and $(p,m)\mapsto pm$. Clearly, $\begin{pmatrix}
	R_1 & M\\
	P & R_2
\end{pmatrix}$ is an associative ring with the usual matrix operations.\\
A Morita context $\begin{pmatrix}
	R_1 & M\\
	P & R_2
\end{pmatrix}$ is said to be trivial if the context products are trivial, that is, $MP=0$ and $PM=0$.

\begin{prop}
	Suppose $R=\begin{pmatrix}
		R_1 & M\\
		P & R_2
	\end{pmatrix}$  is a trivial Morita context. Then, $R$ is NJ-symmetric if and only if $R_1$ and $R_2$ are NJ-symmetric.
\end{prop}
\begin{proof}
	Suppose $R$ is NJ-symmetric. By Proposition \ref{eRe}, $eRe$ is NJ-symmetric. So, $R_1$ and $R_2$ are NJ-symmetric. Conversely, assume that $R_1$ and $R_2$ are NJ-symmetric and $x=\begin{pmatrix}
		a_0 & m_0\\
		p_0 & a_{0_0}
	\end{pmatrix},~ y =\begin{pmatrix}
		b_1 & m_1\\
		p_1 & b_{1_1}
	\end{pmatrix},~ z =\begin{pmatrix}
		c_2 & m_2\\
		p_2 & c_{2_2}
	\end{pmatrix}\in R$ be such that $xyz\in N(R)$. Then, $a_0b_1c_2\in N({R_1})$ and $a_{0_0}b_{1_1}c_{2_2}\in N({R_2})$. Since $R_1$ and $R_2$ are NJ-symmetric rings, $b_1a_0c_2\in J({R_1})$ and $b_{1_1}a_{0_0}c_{2_2}\in J({R_2})$. So, $yxz \in J(R)$, that is, $R$ is NJ-symmetric. 
\end{proof}

Let $R_1$ and $R_2$ be any rings, $M$ a $(R_1,R_2)$-bimodule and $R=\begin{pmatrix}
	R_1 & M\\
	0 & R_2
\end{pmatrix}$, the formal triangular matrix ring. It is well-known that\\ $J(R)=\begin{pmatrix}
	J(R_1) & M\\
	0 & J(R_2)
\end{pmatrix}$.

\begin{cor}
	Let $R_1$ and $R_2$ be any rings and  $M$ a $(R_1,R_2)$- bimodule. Then, $\begin{pmatrix}
		R_1 & M\\
		0 & R_2
	\end{pmatrix}$ is NJ-symmetric if and only if $R_1$ and $R_2$ are NJ-symmetric.
\end{cor}
\begin{cor}\label{upt}
	R is NJ-symmetric if and only if $T_n(R)$ is NJ-symmetric.
\end{cor}
\begin{lemma}\label{nI}
	Let I be a nil ideal of R. If R/I is an NJ-symmetric ring, then so is R.
\end{lemma}

\begin{proof}
	Let $a,b,c\in R$ be such that $abc\in N(R)$. Since $R/I$ is NJ-symmetric ring, $\bar{b}\bar{a}\bar{c}\in J(R/I)$. This implies that $\bar{1}-\bar{b}\bar{a}\bar{c}\bar{x}\in U(R/I)$ for any $x\in R$. Therefore, $1-(1-bacx)y\in I$ for some $y\in R$. Since $I$ is nil, $(1-bacx)$ is right invertible. Similarly, it can be shown that $1-bacx$ is left invertible. So, $1-bacx\in U(R)$. Therefore, $bac\in J(R)$.
\end{proof}

\begin{prop}\label{Rn}
	The following are equivalent:
	\begin{enumerate}
		\item R is NJ-symmetric.
		\item $R_n=
		\left \lbrace
		\left(\begin{array}{lccccr}
			a & a_{12} & \dots  & a_{1(n-1)} & a_{1n}\\
			0 & a & \dots & a_{2(n-1)} & a_{2n}\\
			\vdots & \vdots &\ddots & \vdots & \vdots \\
			0 & 0 & \dots & a &a_{(n-1)n}\\
			0 & 0 & \dots & 0 & a \\
		\end{array}
		\right ):a, a_{ij}\in R,~ i<j\right \rbrace$ is NJ-symmetric.
	\end{enumerate}
\end{prop}

\begin{proof}
	$(1)\implies (2)$. Let $I=\left \lbrace
	\left(\begin{array}{lccccr}
		0 & a_{12} & \dots & a_{1n}\\
		0 & 0 & \dots & a_{2n}\\
		\vdots & \vdots &\ddots & \vdots\\
		0 & 0 & \dots & 0 \\
	\end{array}
	\right ): a_{ij}\in R,~i<j\right \rbrace\subseteq R_n$. Note that $I$ is a nil ideal of $R_n$ and $R_n/I\cong R$. By Lemma \ref{nI}, $R_n$ is NJ-symmetric.
	
	$(2)\implies (1)$. It follows from Proposition \ref{eRe}.  
\end{proof}

Let $A$ be a $(R,R)$-bimodule which is a general ring (not necessarily with unity) in which $(aw)r=a(wr)$, $(ar)w=a(rw)$ and $(ra)w=r(aw)$ hold for all $a,w\in A$ and $r\in R$. Then, \textit{ideal-extension} (also called \textit{Dorroh extension}) $I(R;A)$ of $R$ by $A$ is defined to be the additive abelian group $I(R;A)=R\oplus A$ with multiplication $(r,a)(s,w)=(rs,rw+as+aw)$.

\begin{prop}\label{dor}
	Let $A$ be an $(R,R)$-bimodule which is a general ring (not necessarily with unity) in which $(aw)r=a(wr)$, $(ar)w=a(rw)$ and $(ra)w=r(aw)$ hold for all $a,w\in A$ and $r\in R$. Suppose that for any $a\in A$ there exists $w\in A$ such that $a+w+aw=0$. Then, the following are equivalent:
	\begin{enumerate}
		\item R is NJ-symmetric.
		\item Dorroh extension $S=I(R;A)$ is NJ-symmetric.
	\end{enumerate}
\end{prop}

\begin{proof}
	$(1)\implies (2)$. Let $\alpha=(x,a),~\beta=(y,b),~\gamma=(z,c) \in S$ satisfying  $\alpha\beta\gamma \in N(S)$. Note that $xyz\in N(R)$. This implies that $yxz\in J(R)$ as  $R$ is NJ-symmetric. Firstly, we claim that $0\oplus A\subseteq J(S)$. Let $(0,a)\in 0\oplus A$. For any $(x_1,a_1)\in S$, we have, $(1,0)-(0,a)(x_1,a_1)=(1,-ax_1-aa_1)$. By hypothesis,  $(1,-ax_1-aa_1)(1,a_2)=(1,0)$ for some $a_2\in A$. Hence, $0\oplus A\subseteq J(S)$. Observe that $\beta\alpha\gamma=(yxz,w)$ for some $w\in A$. So, to prove $\beta\alpha\gamma \in J(S)$, it suffices to show $(yxz,0)\in J(S)$. For any $(m,a_3)\in S$, we have, $(1,0)-(yxz,0)(m,a_3)=(1-yxzm,-yxza_3)\in U(S)$, as $(1-yxzm,-yxza_3)=(1-yxzm,0)(1,(1-yxzm)^{-1}(-yxza_3))$, and $1-yxzm\in U(R)$, $(1,(1-yxzm)^{-1}(-yxza_3))=(1,0)+(0,(1-yxzm)^{-1}(-yxza_3))\in U(S)$. Thus, $(yxz,0)\in J(S)$. Therefore, $\beta\alpha\gamma\in J(S)$, that is, $S$ is NJ-symmetric.
	
	$(2)\implies (1)$. Let $a,b, c\in R$ satisfying $abc\in N(R)$. Note that $(a,0)(b,0)(c,0)\in N(S)$. As $S$ is NJ-symmetric ring, $(b,0)(a,0)(c,0)\in J(S)$. Hence, $bac\in J(R)$, that is, $R$ is an NJ-symmetric ring.
\end{proof}

Let $\Psi: R\rightarrow R$ be a ring homomorphism , $R[[x,\Psi]]$ represents the ring of skew formal power series over $R$, that is,  all formal power series in $x$ with coefficients from $R$ and multiplication is defined with respect to the rule $xr=\Psi(r)x$ for all $r\in R$. It is well-known that $J(R[[x,\Psi]])=J(R)+<x>$, $<x>$ is the ideal of $R[[x,\Psi]]$ generated by $x$. Since $R[[x,\Psi]]\cong I(R;<x>)$.
\begin{cor}
	Let $\Psi:R\rightarrow R$ be a ring homomorphism. Then, the following statements are equivalent:
	\begin{enumerate}
		\item R is NJ-symmetric.
		\item $R[[x,\Psi]]$ is NJ-symmetric.
	\end{enumerate}
\end{cor}

\begin{cor}
	The following statements are equivalent:
	\begin{enumerate}
		\item R is NJ-symmetric.
		\item $R[[x]]$ is NJ-symmetric.
	\end{enumerate}
\end{cor}

It is a natural question to ask whether the polynomial ring over an NJ-symmetric ring is NJ-symmetric. In regard to this question, we have the following theorem.

\begin{thm}
	There exists an NJ-symmetric ring $R$ such that $R[x]$
	is not NJ-symmetric.
\end{thm}

\begin{proof}
	Let $K$ be a countable field. By \cite[proof of Lemma 3.7]{olwar}, there exists a nonzero nil algebra $A$ over $K$ satisfying $N^{\ast}(A[x])=0$. Take $R=K+A$. It is evident that $R$ is a local ring with $J(R)=A$. Since local rings are left quasi-duo, $R$ is an NJ-symmetric ring ( by Theorem \ref{qnjs} (\ref{q}) ) and $N^{*}(R[x])=N^{*}(A[x])$. If $R[x]$ is not NJ-symmetric, then the result holds. If $R[x]$ is NJ-symmetric, then  we show that $(R[x])[y]$ is not NJ-symmetric. Assume, if possible, that $(R[x])[y]$ is NJ-symmetric. By \cite[Theorem 1]{radpol}, $J((R[x])[y])=I[y]$ for some nil ideal $I$ of $R[x]$, where $I=N^{\ast}(R[x])=N^{*}(A[x])=0$. Therefore, $J((R[x])[y])=0$. So, $(R[x])[y]$ is symmetric, which further implies that $R[x]$ is a symmetric ring. As symmetric rings are 2-primal,  $N(R[x])=N_*(R[x])$. But, this is a contradiction to the fact that $0\neq N(R)=A\subseteq N(R[x])$ and $N_*(R[x])\subseteq N^*(R[x])=N^*(A[x])=0$.
\end{proof}



\section{Declarations}\label{sec12}

\textbf{Conflict of interest} The authors have no conflict of interest to declare that are
relevant to the content of this study.

\end{document}